\documentclass[12pt,leqno]{article}
\usepackage{amssymb, amscd, amsmath, amsthm}
\usepackage{dsfont}
\allowdisplaybreaks

\def\beq{\begin{equation}}
\def\eeq{\end{equation}}

\theoremstyle{definition}
\newtheorem{definition}{Definition}
\newtheorem{problem}{Problem}
\newtheorem{example}{Example}
\newtheorem*{rema}{Remark}

\theoremstyle{plain}
\newtheorem{theorem}{Theorem}
\newtheorem{claim}{Claim}

\newtheorem{lemma}{Lemma}

\newtheorem{corollary}{Corollary}

\newtheorem{proposition}{Proposition}

\numberwithin{equation}{section}
\numberwithin{proposition}{section}
\numberwithin{definition}{section}
\numberwithin{theorem}{section}
\numberwithin{problem}{section}
\numberwithin{example}{section}
\numberwithin{claim}{section}
\numberwithin{fact}{section}
\numberwithin{lemma}{section}
\numberwithin{conjecture}{section}
\numberwithin{corollary}{section}
\numberwithin{observation}{section}

\begin{document}

\title{On strengthenings of the intersecting shadow theorem}

\author{by P. Frankl and G. O. H. Katona
%%\thanks{Research partially supported by the Ministry of Education and Science of the Russian Federation in the framework of MegaGrant no 075-15-2019-1926.}
\\
R\'enyi Institute, Budapest, Hungary}

\date{}
\maketitle

\begin{abstract}
Let $n > k > t \geq j \geq 1$ be integers.
Let $X$ be an $n$-element set, ${X\choose k}$ the collection of its $k$-subsets.
A family $\mathcal F \subset {X\choose k}$ is called $t$-intersecting if $|F \cap F'| \geq t$ for all $F, F' \in \mathcal F$.
The $j$'th shadow $\partial^j \mathcal F$ is the collection of all $(k - j)$-subsets that are contained in some member of~$\mathcal F$.
Estimating $|\partial^j \mathcal F|$ as a function of $|\mathcal F|$ is a widely used tool in extremal set theory.
A classical result of the second author (Theorem \ref{th:1.3}) provides such a bound for $t$-intersecting families.
It is best possible for $|\mathcal F| = {2k - t\choose k}$.

Our main result is Theorem \ref{th:1.4} which gives an asymptotically optimal bound on $|\partial^j \mathcal F| / |\mathcal F|$ for $|\mathcal F|$  slightly larger, e.g., $|\mathcal F| > \frac32 {2k - t\choose k}$.
We provide further improvements for $|\mathcal F|$ very large as well.
\end{abstract}

\section{Introduction}
\label{sec:1}

Throughout the paper $n, k, t$ are positive integers, $n > k > t$.
Let $[n] = \{1, 2, \dots, n\}$ be the standard $n$-element set and ${[n]\choose k}$ the collection of all its $k$-subsets.
For a family $\mathcal F \subset {[n]\choose k}$ and $0 < j < k$ define the \emph{$j$'th shadow}
$\partial^j \mathcal F = \left\{G \in {[n]\choose k - j} : \exists F \in \mathcal F, G \subset F\right\}$.

Estimating the minimum possible size, $|\partial^j \mathcal F|$ in function of $|\mathcal F|$ has proved to be one of the most important tools of extremal set theory.
As a matter of fact, the first paper written on this subject, due to Sperner, is heavily relying on such a bound.

\begin{proposition}[Sperner \cite{S}]
\label{prop:1.1}
Suppose that $\emptyset \neq \mathcal F \subset {[n]\choose k}$, $0 < j < k$.
Then
\beq
\label{eq:1.1}
|\partial^j \mathcal F| / |\mathcal F| \geq {n\choose k - j} \biggm/ {n\choose k}
\eeq
with equality holding iff $\mathcal F = {[n]\choose k}$.
\end{proposition}

The classical Kruskal--Katona Theorem (\cite{Kr}, \cite{Ka2}) determines the minimum of $|\partial^j \mathcal F|$, given $|\mathcal F|$.

For $j = 1$ the notation $\partial\mathcal F$ is common and $\partial\mathcal F$ is called the \emph{immediate shadow}.

\setcounter{definition}{1}

\begin{definition}
\label{def:1.2}
Let $0 \leq \ell < k$, $\mathcal F \subset {[n]\choose k}$.
Define the \emph{$\ell$-shadow} $\sigma_\ell(\mathcal F)$ by
$$
\sigma_\ell(\mathcal F) = \left\{G \in {[n]\choose \ell} : \exists F \in \mathcal F, \, G \subset F\right\}.
$$
Note that $\partial \mathcal F = \sigma_{k - 1} (\mathcal F)$ and $\partial^{k - \ell} \mathcal F = \sigma_\ell(\mathcal F)$.
\end{definition}

One of the most widely investigated properties in extremal set theory is the $t$-intersecting property.
For $t \geq 1$, $\mathcal F$ is said to be \emph{$t$-intersecting} if $|F \cap F'| \geq t$ for all $F, F' \in \mathcal F$.
For $t = 1$, the term \emph{intersecting} is used as well.

A widely used result of the second author shows that $|\partial^j \mathcal F| \geq |\mathcal F|$ for $0 < j \leq t$ provided that $\mathcal F$ is $t$-intersecting.

\setcounter{theorem}{2}

\begin{theorem}[Intersecting Shadow Theorem \cite{Ka1}]
\label{th:1.3}
Suppose that $\emptyset \neq \mathcal F \subset {[n]\choose k}$, $\mathcal F$ is $t$-intersecting, $k - t \leq \ell < k$.
Then
\beq
\label{eq:1.2}
\bigl|\sigma_\ell(\mathcal F)\bigr| \bigm/|\mathcal F| \geq {2k - t\choose \ell} \biggm/{2k - t\choose k}
\eeq
with strict inequality unless $\mathcal F = {Y\choose k}$ for some $2k - t$-element set~$Y$.
\end{theorem}

Note that for $n \leq 2k - t$ the inequality \eqref{eq:1.2} can be deduced from Sperner's bound \eqref{eq:1.1}.
However for fixed $k$ and $n$ tending to infinity the RHS of \eqref{eq:1.1} tends to~$0$ while the RHS of \eqref{eq:1.2} is at least~$1$.
To be more exact, for $\ell = k - 1$ its value is $k\bigm/(k - t + 1)$.
For $t \geq 2$ this is strictly larger than~$1$.
Our first result gives a further improvement provided that $|\mathcal F| \geq \left(1 + \frac{t - 1}{k + t}\right) {2k - t\choose k}$.

\begin{theorem}
\label{th:1.4}
Suppose that $\mathcal F \subset {[n]\choose k}$, $\mathcal F$ is $t$-intersecting, $1 \leq j < t < k$, $|\mathcal F| \geq {2k - t\choose k} \left(1 + \frac{t - j}{k + t + 1 - j}\right)$.
Then
\beq
\label{eq:1.3}
\bigl|\partial^j \mathcal F| \bigm/ |\mathcal F| \geq {2(k - 1) - t\choose k - 1 - j} \biggm/ {2(k - 1) - t\choose k - 1}.
\eeq
\end{theorem}

Let us mention that the requirement on $|\mathcal F|$ is relatively weak, e.g., it is weaker than $|\mathcal F| \geq \frac32 {2k - t\choose k}$.
For $j = 1$, the most widely used case, \eqref{eq:1.3} reduces to
$$
|\partial \mathcal F| / |\mathcal F| \geq \frac{k - 1}{k - t}.
$$

At first sight it might appear to be only a small improvement with respect to $\frac{k}{k - t + 1}$, coming from \eqref{eq:1.2}.
However, for $k$ and $t$ fixed the difference is substantial.
Most importantly, the new bound is essentially best possible.

\setcounter{example}{4}
\begin{example}
\label{ex:1.5}
Fix $k > t > 2$ and an integer $s$, $0 \leq s < k - t - 1$.
Define $\mathcal A = \left\{A \in {[2k - t]\choose k} : |A \cap [k - 1 + s]| \geq t + s\right\}$, \ $\mathcal B = \Bigl\{B \in {[n]\choose k}, B_0 \cup \{x\}$, $B_0 \in {[k - 1 + s]\choose k - 1}, x \in [2k - t + 1, n]\Bigr\}$.
Set $\mathcal F = \mathcal A \cup \mathcal B$.
Then $\mathcal F$ is $t$-intersecting.
\end{example}

\setcounter{proposition}{5}
\begin{proposition}
\label{prop:1.6}
For a proper choice of $s$ and $n$, Example \ref{ex:1.5} shows that \eqref{eq:1.3} does not hold for $k > k_0(j)$ even if
$$
|\mathcal F| = \left(1 + \frac{j(t - j) s(s - 1) \cdot \ldots \cdot (s - j + 1)}{(k - 1)^{j + 2}} - o(1)\right) {2k - t \choose k}.
$$
\end{proposition}

The paper is organized as follows.
In Section \ref{sec:2} we review some results concerning shifting and shifted families.
Then we prove Theorem \ref{th:2.10} concerning shadows.
In Section \ref{sec:3} we prove Theorem \ref{th:1.4}, in the very short Section \ref{sec:4} the proof of Proposition \ref{prop:1.6} is provided.

In Section \ref{sec:5} we introduce the notion of a semistar and prove a best possible lower bound on the shadow of $t$-intersecting semistars (Theorem \ref{th:5.5}).
In Section~\ref{sec:6} along with some structural results we prove the best possible bound $\bigl|\partial^j \mathcal F\bigr| > {t\choose j}|\mathcal F|$ for families satisfying $|\mathcal F| > (t + 2){n - t - 1\choose k - t - 1}$, $n > n_0(k, t)$ in a more precise form.

Section \ref{sec:7} contains some more general results.

\section{Preliminaries}
\label{sec:2}

Let $(a_1, \dots, a_k)$ denote the $k$-element set $\{a_1, \dots, a_k\}$ where we know that $a_1 < \ldots < a_k$.
Let us define $\prec$, the \emph{shifting partial order} by setting
$$
(a_1, \dots, a_k) \prec (b_1, \dots, b_k)\ \ \ \ \text{ iff } \ \ \ a_i \leq b_i \ \ \ \text{ for } \ \ 1 \leq i \leq k.
$$

\begin{definition}
\label{def:2.1}
The family $\mathcal F$ is called \emph{shifted} if $(a_1, \dots, a_k) \prec (b_1, \dots, b_k)$ and $(b_1, \dots, b_k) \in \mathcal F$ always imply $(a_1, \dots, a_k) \in \mathcal F$.
\end{definition}

In their seminal paper \cite{EKR}, Erd\H{o}s, Ko and Rado defined a simple operation on families of sets called shifting.
Repeated application of this operation eventually transforms a family into a shifted family.
Erd\H{o}s, Ko and Rado showed that shifting maintains the $t$-intersecting property.
In \cite{Ka1} it is shown that shifting never increases the $\ell$-shadow.
Consequently, it is sufficient to prove Theorem \ref{th:1.4} for shifted families.

On the other hand, shifted $t$-intersecting families have some nice properties.

\setcounter{proposition}{1}
\begin{proposition}[\cite{F78}]
\label{prop:2.2}
Suppose that $\mathcal F \subset {[n]\choose k}$ is shifted and $t$-intersecting.
Then for every $F \in \mathcal F$ there exists an integer $h$, $0 \leq h \leq k - t$ such that
\beq
\label{eq:2.1}
\bigl|F \cap [t + 2h]\bigr| \geq h + t.
\eeq
\end{proposition}

In \cite{F78} the following families were defined:
$$
\mathcal A_h(n, k, t) = \left\{A \in {[n]\choose k} : \bigl|A \cap [t + 2h]\bigr| \geq h + t\right\}.
$$
It is easy to see that $\mathcal A_h(n, k, t)$ is always $t$-intersecting.

In \cite{F78} it was conjectured that for $n \geq 2k - t$,
\beq
\label{eq:2.2}
|\mathcal F| \leq \max \left\{\bigl|\mathcal A_h (n, k, t) \bigr| : 0 \leq h \leq k - t\right\}.
\eeq

In \cite{FF2} \eqref{eq:2.2} was proved for a wide range.
However, it was not before the seminal paper of Ahlswede and Khachatrian \cite{AK1} that \eqref{eq:2.2} was established in its integrity.

It is easy to check that for $k$ and $t$ fixed
$$
\lim_{n \to \infty} \bigl|\partial^j \mathcal A_{k - t - 1}(n, k, t) \bigr|\bigm/\bigl|\mathcal A_{k - t - 1} (n, k, t)\bigr| =
{2(k\! -\! 1)\! -\! t\choose k - 1 - j}\! \biggm/\! {2(k\! -\! 1) \! -\! t\choose k - 1}
$$
which shows that \eqref{eq:1.3} is essentially best possible.

Based on Proposition \ref{prop:2.2} one can define the following relaxation of the $t$-intersecting property.

\setcounter{definition}{2}
\begin{definition}
\label{def:2.3}
The family $\mathcal F \subset {[n] \choose k}$ is said to be \emph{pseudo $t$-intersecting} if for every $F \in \mathcal F$ and some $h$, $0 \leq h \leq k - t$, \eqref{eq:2.1} holds.
\end{definition}

It was shown in \cite{F91} that \eqref{eq:1.2} holds for pseudo $t$-intersecting families as well.

We need some more definitions.

Let $\mathcal F \subset {[n] \choose k}$ be pseudo $t$-intersecting.
Define the \emph{width} $w = w_t(\mathcal F)$ as the minimum integer such that for every $F \in \mathcal F$ \eqref{eq:2.1} holds for some $h$, $0 \leq h \leq w$.
From Definition \ref{def:2.3} it is clear that $w_t(\mathcal F)$ exists and $w_t(\mathcal F) \leq k - t$.
However, in certain situations it needs to be smaller.
For example, define $\mathcal F_{\text{out}} = \mathcal F \setminus {[2k - t]\choose k}$.
For $F \in \mathcal F_{\text{out}}$, $|F \cap [2k - t]| < k$ implies $w_t(\mathcal F_{\text{out}}) \leq k - t - 1$.
This will be very important for our proofs.

\begin{definition}
\label{def:2.4}
Let $\mathcal F \subset {X \choose k}$ be pseudo $t$-intersecting and $w = w_t(\mathcal F)$.
For $F \in \mathcal F$ define its \emph{height} $h(F)$ as
$$
h(F) = \max\left\{h: 0 \leq h \leq w, \bigl|F \cap [t + 2h]\bigr| \geq t + h\right\}.
$$
\end{definition}

\setcounter{claim}{4}

\begin{claim}
\label{cl:2.5}
If $h(F) < w$ then
\beq
\label{eq:2.3}
\bigl|F \cap [t + 2h(F)]\bigr| = t + h(F).
\eeq
\end{claim}

\begin{proof}
Should $\bigl|F \cap [t + 2h(F)]\bigr| \geq t + h(F) + 1$ hold, we conclude
$$
\bigl|F \cap [t + 2(h(F) + 1)]\bigr| \geq t + h(F) + 1,
$$
contradicting the maximal choice of $h(F)$.
\end{proof}

Let us define the tail $T = T(F)$ for $F \in \mathcal F$ by $T(F) = F \setminus [t + 2h(F)]$.
In view of \eqref{eq:2.3},
\beq
\label{eq:2.4}
|T(F)| = k - t - h(F) \ \ \ \text{ holds if }\ \ \ h(F) < w_t(\mathcal F).
\eeq
If $h(F) = w_t(\mathcal F)$ then either \eqref{eq:2.4} holds or
$$
|T| < k - t - h(F).
$$

\setcounter{definition}{5}
\begin{definition}
\label{def:2.6}
For $0 < j \leq t$ and $F \in \mathcal F$ let us define the \emph{restricted $j$'th shadow} $\partial_R^j F = \left\{G \in {F \choose k - j} : T \subset G\right\}$.
In human language $G$ is obtained from $F$ by arbitrarily deleting $j$ vertices from $F \setminus T$.
\end{definition}

\setcounter{claim}{6}
\begin{claim}
\label{cl:2.7}
If $h(F) < w_t(\mathcal F)$ and $G \in \partial_R^j F$ then {\rm (i)} and {\rm (ii)} hold.

\hspace*{3.4pt}{\rm (i)} $\bigl|G \cap [t + 2h(F)]\bigr| = t - j + h(F)$,

\smallskip
{\rm (ii)} $\bigl|G \cap [t + 2h]\bigr| < t - j + h$ for $h(F) < h \leq w_t(\mathcal F)$.\hfill $\square$
\end{claim}

Applying this claim we infer

\setcounter{corollary}{7}
\begin{corollary}
\label{cor:2.8}
Suppose that $F, F' \in \mathcal F$, $h(F) < h(F')$.
Then
\beq
\label{eq:2.5}
\partial_R^j F \cap \partial_R^j F' = \emptyset.
\eeq
\end{corollary}

\begin{proof}
Using (i) and (ii)
$$
\bigl|G \cap [t + 2h(F')]\bigr| < \bigl|G' \cap [t + 2h(F')]\bigr|
$$
follows for $G \in \partial_R^j F$ and $G' \in \partial_R^j F'$.
\end{proof}

Note that \eqref{eq:2.5} is immediate also if $h(F) = h(F')$ but $T(F) \neq T(F')$.
Define $\mathcal T = \bigl\{T \subset [n] : \exists F \in \mathcal F, T(F) = T\bigr\}$.
For $T \in \mathcal T$ define $\mathcal F_T = \{F \in \mathcal F : T(F) = T\}$ and $\overline{\mathcal F}_T = \bigl\{F \setminus T : F \in \mathcal F_T\bigr\}$.
This permits to define the restricted $j$'th shadow of $\mathcal F_T$:
$$
\partial_R^j \mathcal F_T = \bigcup_{F \in \mathcal F_T} \partial_R^j F.
$$

The next lemma is the core of the proofs.

\setcounter{lemma}{8}
\begin{lemma}
\label{lem:2.9}
Suppose that $\mathcal F$ is pseudo $t$-intersecting, $0 < j \leq t$.
Then
$\mathcal F = \bigcup\limits_{T \in \mathcal T} \mathcal F_T$ is a partition, and
\beq
\label{eq:2.6}
\bigl|\partial^j \mathcal F\bigr| \geq \sum_{T \in \mathcal T} \bigl|\partial_R^j \mathcal F_T\bigr|.
\eeq
\end{lemma}

\begin{proof}
The first part is trivial.
To show the second one we need to prove for $T, T'\in \mathcal T$, $T \neq T'$,
$$
\partial_R^j \mathcal F_T \cap \partial_R^j \mathcal F_{T'} = \emptyset.
$$
This follows from \eqref{eq:2.5} unless both $F$ and $F'$ with $T(F) = T$ and $T(F') = T'$ satisfy $h(F) = h(F') = w = w_t(\mathcal F)$.
(Actually, by \eqref{eq:2.3} these are equivalent to $|T|, |T'| \leq k - t - w$.)
In this case $T = F \setminus [t + 2w]$, $T' = F' \setminus [t + 2w]$ imply $\partial_R^j \mathcal F_T \cap \partial_R^j \mathcal F_{T'} = \emptyset$.
\end{proof}

With this preparation the next theorem is easy to prove.

\setcounter{theorem}{9}
\begin{theorem}
\label{th:2.10}
Let $\mathcal F \subset {[n] \choose k}$ be a shifted pseudo $t$-intersecting of width $w = w_t(\mathcal F)$.
Then for every $0 < j \leq t$,
\beq
\label{eq:2.7}
\bigl|\partial_R^j \mathcal F\bigr| \geq |\mathcal F| {t + 2w\choose t - j + w} \biggm/ {t + 2w \choose t + w}.
\eeq
\end{theorem}

\begin{proof}
Let $\mathcal T$ be the family of possible tails for $\mathcal F$.
In view of Lemma \ref{lem:2.9} it is sufficient to show
\beq
\label{eq:2.8}
\bigl|\partial_R^j \mathcal F_T\bigr| \geq |\mathcal F_T| {t + 2w\choose t - j + w} \biggm/ {t + 2w\choose t + w}.
\eeq
Recall that $\overline{\mathcal F}_T = \{F \setminus T : F \in \mathcal F_T\}$.
If $|T| \geq k - t - w$ then $\overline{\mathcal F}_T \subset {[t + 2(k - t - |T|)]\choose k - |T|}$ and
$\bigl|\partial_R^j \mathcal F_T\bigr| = \bigl|\partial^j \overline{\mathcal F}_T\bigr|$.

If $|T| < k - t - w$ then $\overline{\mathcal F}_T \subset {[t + 2w]\choose k - |T|}$ and again
$\bigl|\partial_R^j \mathcal F_T\bigr| = \bigl|\partial_j \overline{\mathcal F}_T\bigr|$.
In the first case $t + 2(k - t - |T|) = 2(k - |T|) - t$, showing that $\overline{\mathcal F}_T$ is $t$-intersecting.
In the second case $t + 2w > 2(k - |T|) - t$ by $w + |T| < k - t$, that is $\overline{\mathcal F}_T$ is $(t + 1)$-intersecting.
However, the desired bound readily follows using \eqref{eq:1.1} and the next proposition.

\setcounter{proposition}{10}
\begin{proposition}
\label{prop:2.11}
Let $0 < j < t$, $0 \leq h < w$ and $1 \leq r \leq w$, then the following two inequalities hold.

\hspace*{3.4pt}{\rm (i)} $\displaystyle{t + 2h\choose t + h - j} \biggm/ {t + 2h\choose t + h} > {t + 2w\choose t + w - j} \biggm/ {t + 2w\choose t + w}$,

\smallskip
{\rm (ii)} $\displaystyle{t + 2w\choose t + w - j + r} \biggm/ {t + 2w\choose t + w + r} > {t + 2w\choose t + w - j} \biggm/ {t + 2w\choose t + w}$.
\end{proposition}

\begin{proof}
Let $f(h)$ denote the LHS of (i).
That is, $f(h) = \prod\limits_{1 \leq i \leq j} \frac{t + h - j + i}{h + i} = \prod\limits_{1\leq i \leq j} \left(1 + \frac{t - j}{h + i}\right)$.
Since $1 + \frac{t - j}{h + i}$ is a strictly monotone decreasing function of $h$, \,$f(h) > f(w)$ follows.

To prove (ii) let $g(r)$ be the LHS, i.e.,
$$
g(r) = \prod_{1 \leq i \leq j} \frac{t - j + w + i + r}{w + i - r}.
$$
Since $\frac{a + r}{b - r}$ is a strictly monotone increasing function of $r$ (for $a > 0$, $b > r$), $g(r) > g(0)$ and thereby (ii) follows.
\end{proof}

This concludes the proof of Theorem \ref{th:2.10} as well.
\end{proof}

\section{The proof of Theorem \ref{th:1.4}}
\label{sec:3}

Let $\mathcal F \subset {[n] \choose k}$ be a shifted $t$-intersecting family, $t \geq 2$.
If $w_t(\mathcal F) \leq k - t - 1$ then for every $1 \leq j < t$, from Theorem \ref{th:2.10} we infer
$$
\bigl|\partial^j \mathcal F\bigr| \geq |\mathcal F| {t + 2(k - t - 1)\choose k - 1 - j} \biggm/ {t + 2(k - t - 1)\choose k - 1}
$$
proving \eqref{eq:1.3}.

From now on we suppose $w_t(\mathcal F) = k - t$ and fix an $A = (a_1,\dots, a_k) \in \mathcal F$ such that
\beq
\label{eq:3.1}
\bigl|A \cap [t + 2h]\bigr| \leq t + h - 1 \ \ \ \text{ for } \ \ \ 0 \leq h < k - t.
\eeq
Applying \eqref{eq:2.1} to $A$ yields $\bigl|A \cap [t + 2(k - t)]\bigr| = k$, i.e., $A \in {[2k - t]\choose k}$.
Our plan for proving \eqref{eq:1.3} is the following.
We partition $\mathcal F$ into two families $\mathcal F_{\text{in}}$ and $\mathcal F_{\text{out}}$ where
$\mathcal F_{\text{in}} = \mathcal F \cap {[2k - t]\choose k}$, $\mathcal F_{\text{out}} = \mathcal F \setminus \mathcal F_{\text{in}}$.
Then we show that
\beq
\label{eq:3.2}
\partial^j \mathcal F_{\text{in}} \cap \partial_R^j \mathcal F_{\text{out}} = \emptyset
\eeq
and thereby
\beq
\label{eq:3.3}
\bigl|\partial^j \mathcal F| \geq \bigl|\partial^j \mathcal F_{\text{in}}\bigr| + \bigl|\partial_R^j \mathcal F_{\text{out}}\bigr|.
\eeq

For the first term on the RHS we use \eqref{eq:1.2} with $\ell = k - j$.
As for the second, we prove a stronger inequality
\beq
\label{eq:3.4}
\bigl|\partial_R^j \mathcal F_{\text{out}}\bigr| \geq \bigl|\mathcal F_{\text{out}}\bigr| {t + 1 + 2(k - t - 2)\choose k - 1 - j}\biggm/ {t + 1 + 2(k - t - 2)\choose k - 1}.
\eeq
Defining $\alpha = \alpha(k, t, j)$ and $\beta = \beta(k, t, j)$ by
$$
\alpha = \frac{{t + 2(k - t - 1)\choose k - 1 - j}}{{t + 2(k - t - 1)\choose k - 1}} - \frac{{t + 2 (k - t)\choose k - j}}{{t + 2(k - t)\choose k}}, \ \ \ \
\beta = \frac{{t + 1 + 2(k - t - 2)\choose k - 1 - j}}{{t + 1 + 2(k - t - 2)\choose k - 1}} - \frac{{t + 2 (k - t)\choose k - j}}
{{t + 2(k - t)\choose k}},
$$
\eqref{eq:3.3} and \eqref{eq:3.4} imply
\beq
\label{eq:3.5}
\bigl|\partial^j \mathcal F\bigr| \geq \left(\bigl|\mathcal F_{\text{in}}\bigr| + \bigl|\mathcal F_{\text{out}}\bigr|\right)
\frac{{t + 2(k - t)\choose k - j}}{{t + 2(k - t)\choose k}} + \beta \bigl|\mathcal F_{\text{out}}\bigr|.
\eeq

Finally we show that the assumption on $|\mathcal F|$ implies
$$
\bigl|\mathcal F_{\text{out}}\bigr| \geq |\mathcal F| - {2k - t\choose k} \geq \frac{\alpha}{\beta} |\mathcal F|.
$$
Plugging this into \eqref{eq:3.5} yields
$$
\bigl|\partial^j \mathcal F\bigr| \geq \left(\frac{{t + 2(k - t)\choose k - j}}{{t + 2(k - t)\choose k}} + \alpha\right)
|\mathcal F| = \frac{{t + 2(k - t - 1)\choose k - 1 - j}}{{t + 2(k - t - 1)\choose k - 1}} |\mathcal F|,\ \  \text{ as desired.}
$$

Let us now execute this plan.
\eqref{eq:3.2} is essentially trivial.
If $G \in \partial^j \mathcal F_{\text{in}}$ then $G \subset [2k - t]$.
For $F \in \mathcal F_{\text{out}}$, $F \not\subset [2k - t]$ and the definition of the restricted shadow imply that $G'\not\subset [2k - t]$ for each $G' \in \partial_R^j F$.
Thus $G \neq G'$.

To prove \eqref{eq:3.4} let us show:

\begin{proposition}
\label{prop:3.1}
The family $\mathcal F_{\text{\rm out}}$ is pseudo $(t + 1)$-intersecting and $w_{t + 1}(\mathcal F_{\text{\rm out}}) \leq k - t - 2$.
\end{proposition}

\begin{proof}
Define the two sets $E$ and $D$ as follows.
\begin{align*}
E &= (1, 2, \dots, t - 1, t + 1, t + 3, \dots, 2k - t - 3, 2k - t - 1, 2k - t),\\
D &= (1, 2, \dots, t, t + 2, t + 4, \dots, 2k - t - 2, 2k - t + 1).
\end{align*}
Note that $E \cap D = [t - 1]$.

Let us show that $E \prec A$, implying $E \in \mathcal F$.
$i \leq a_i$ is trivial for $1 \leq i < t$.
As to $a_{t + h}$, $0 \leq h < k - t$, \eqref{eq:3.1} implies $t + 2h + 1 \leq a_{t + h}$.
Finally, using this inequality for $h = k - t - 1$ gives $2k - t - 1 \leq a_{k - 1}$ implying $2k - t \leq a_k$.
By shiftedness $E \in \mathcal F$.
On the other hand the $t$-intersecting property and $|D \cap E| = t - 1$ imply $D \notin \mathcal F$.

Choose an arbitrary $B = (b_1, \dots, b_k) \in \mathcal F_{\text{out}}$.
As $\mathcal F$ is shifted, $D \prec B$ cannot hold.
Note that for $1 \leq i \leq t$, $i \leq b_i$.
Also, $B \notin \mathcal F_{\text{in}}$ implies $2k - t + 1 \leq b_k$.
Therefore there exists a $g$, $0 \leq g \leq k - t - 2$ such that $b_{t + 1 + g}$ is strictly smaller than the corresponding element of~$D$.
That is,
$$
b_{t + 1 + g} \leq t + 1 + 2g.
$$
Equivalently
$$
\bigl|B \cap [t + 1 + 2g]\bigr| \geq t + 1 + g
$$
proving the pseudo $(t + 1)$-intersecting property.
Also, $g \leq k - t - 2$ implies $w_{t + 1}(\mathcal F_{\text{out}}) \leq k - t - 2$ as well.
\end{proof}

Now \eqref{eq:3.4} follows by applying Theorem \ref{th:2.10} with $t$ replaced by $t + 1$.

Let us compute $\alpha$ and $\beta$.
$$
\frac{{2k - t - 2\choose k - 1 - j}}{{2k - t - 2\choose k - 1}} \Bigg/ \frac{{2k - t\choose k - j}}{{2k - t\choose k}} = \frac{(k - j)(k - t + j)}{k (k - t)} = 1 + \frac{j(t - j)}{k(k - t)}.
$$
Thus
\begin{align}
\label{eq:3.6}
\alpha &= \frac{j(t - j)}{k(k - t)} \cdot \frac{{2k - t\choose k - j}}{{2k - t\choose k}}.\\
& \frac{{2k - t - 3\choose k - 1 - j}}{{2k - t - 3\choose k - 1}} \Biggm/ \frac{{2k - t\choose k - j}}{{2k - t\choose k}}
=
\frac{(k - j)(k - t + j)(k - t + j - 1)}{k(k - t)(k - t - 1)}\nonumber\\
&= 1 + \frac{j(k^2 - t^2 - t) - j^2(k - 2t - 1) - j^3}{k(k - t)(k - t - 1)}.
\nonumber
\end{align}
Thus
$$
\beta = \frac{j(k^2 - t^2 - t) - j^2(k - 2t - 1) - j^3}{k(k - t)(k - t - 1)} \cdot \frac{{2k - t\choose k - j}}{{2k - t\choose k}}.
$$
Consequently,
\begin{align*}
\frac{\beta}{\alpha} &= \frac{k^2 - t^2 - t - j(k - 2t + 1) - j^2}{(t - j)(k - t - 1)}\\
&= \frac{k + t + 1 - j}{t - j} + \frac{t + 1 + (t - j)j}{(t - j)(k - t - 1)} > \frac{k + t + 1 - j}{t - j}.
\end{align*}
We proved
$$
\frac{\alpha}{\beta} < \frac{t - j}{k + t + 1 - j}.
$$
On the other hand the assumption of Theorem \ref{th:1.4} was
$$
|\mathcal F| \geq {2k - t\choose k} \left(1 + \frac{t - j}{k + t + 1 - j}\right)
$$
implying
$$
\bigl|\mathcal F_{\text{out}}\bigr|\bigm/ |\mathcal F| \geq \frac{t - j}{k + t + 1 - j} > \frac{\alpha}{\beta},
$$
concluding the proof. \hfill $\square$

\section{The proof of Proposition \ref{prop:1.6}}
\label{sec:4}

First of all note that
$$
\left|{[2k - t]\choose k}\setminus |\mathcal A|\right| = \sum_{0 \leq i < t + s} {k - 1 + s\choose i} {k + 1 - s - t\choose k - i} = o\left({2k - t\choose k}\right)
$$
for fixed $s, t$ as $k \to \infty$.

Let us compute the size of $\partial^j \mathcal B \setminus {[2k - t]\choose k - j}$.
For a fixed $x \in [2k - t + 1, n]$, $\{x\} \cup B_0 \in \mathcal B$ iff
$B_0 \in {[k - 1 + s]\choose k - 1}$.
Thus the sets $D \in \left(\partial^j B\setminus {[2k - t]\choose k}\right)$ are of the form $\{x\} \cup D_0$ with $D_0 \in {[k - 1 + s]\choose k - 1 - j}$.
Thus
$$
\bigl|\partial^j \mathcal F\bigr| \leq {2k - t\choose k - j} + (n - 2k + t) {k - 1 + s \choose s + j}.
$$
Comparing this with
$$
|\mathcal F| = (1 - o(1)) {2k - t\choose k} + (n - 2k + t) {k - 1 + s\choose k - 1}
$$
and recalling the definition of $\alpha$ (cf.\ Section~\ref{sec:3}), we see that
$\bigl|\partial^j \mathcal F\bigr| / |\mathcal F| < (1 + \alpha) {2k - t\choose k - t}$ as long as
$$
|\mathcal B| < \frac{\alpha{k - 1 + s\choose s}}{{k - 1 + s\choose s + j}} {2k - t\choose k}(1 - o(1)).
$$
Noting ${k - 1 + s\choose s} \bigm/ {k - 1 + s\choose s + j} = \prod\limits_{0 \leq i < j} \frac{k - 1 - i}{s - i} < \frac{(k - 1)^j}{s(s - 1)\cdot \ldots \cdot (s - j + 1)}$ and
$\alpha > \frac{j(t - j)}{k(k - t)}$ we see that
$$
|\mathcal B| < \frac{j(t - j) s(s - 1)\cdot \ldots \cdot (s - j + 1)}{(k - 1)^{j + 2}} {2k - t\choose k} (1 - o(1))
$$
is fine.
Setting $\varepsilon(k) = \frac{j(t - j)s(s - 1) \cdot \ldots \cdot (s - j + 1)}{(k - 1)^{j + 2}}$ we get
$|\mathcal F| = (1 + \varepsilon(k) - o(1)) {2k - t\choose k}$.

\hfill $\square$

\section{The shadow of stars and semistars}
\label{sec:5}

The most important result concerning intersecting families is the Erd\H{o}s--Ko--Rado Theorem.

\begin{theorem}[\cite{EKR}]
\label{th:5.1}
Suppose that $n \geq n_0(k, t)$, $\mathcal F \subset {[n]\choose k}$ is $t$-intersecting, $k > t > 0$.
Then
\beq
\label{eq:5.1}
|\mathcal F| \leq {n - t\choose k - t}.
\eeq
\end{theorem}

As to the bound $n_0(k,t)$, its exact value is $(k - t + 1)(t + 1)$.
For $t = 1$ it was proved already by Erd\H{o}s, Ko and Rado.
For $t \geq 15$ it was proved by the first author (\cite{F78}).
Finally Wilson \cite{W} showed it by a proof using eigenvalues for $2 \leq t \leq 14$ (the proof is valid for all $t$).

The \emph{full $t$-star}, $\mathcal A_0(n, k, t) = \left\{A \in {[n]\choose k} : [t] \subset A\right\}$ shows that \eqref{eq:5.1} is best possible.
Let us note that for $n = (k - t + 1)(t + 1)$, $\bigl|\mathcal A_0(n, k, t)\bigr| = \bigl|\mathcal A_1(n, k, t)\bigr|$ and for $t \geq 2$ up to isomorphism these are the only families achieving equality in \eqref{eq:5.1}.

Let us mention that the Intersecting Shadow Theorem implies $|\mathcal F| \leq \bigl|\partial^t\mathcal F| \leq {n\choose k - t}$ for all $n \geq 2k - t$.
Very recently the first author \cite{F20} showed the slightly stronger universal bound
\beq
\label{eq:5.2}
|\mathcal F| \leq {n - 1\choose k - t}\ \ \ \text{ for all }\ \ n > 2k - t, \ \ \ \mathcal F \ \text{ is $t$-intersecting}.
\eeq

\setcounter{definition}{1}
\begin{definition}
\label{def:5.2}
If $C \subset F$ holds for all $F \in \mathcal F$ with a $t$-set $C$ then $\mathcal F$ is called a \emph{$t$-star}.
If for some $(t + 1)$-element set $D$, $|F \cap D| \geq t$ holds for all $F \in \mathcal F$ then $\mathcal F$ is called a \emph{$t + 1$-semistar}.
When the value of $t$ is clear from the context, we say for short that $\mathcal F$ is a \emph{star} or \emph{semistar}.
\end{definition}

Let us note that the family $\mathcal A_0(n, k, t) \cup \mathcal A_1(n, k, t)$ is a semistar with $D = [t + 1]$.

Let us fix $n, k, t$, $t \geq 2$ and use the shorthand notation $\mathcal A_0$, $\mathcal A_1$.

\setcounter{proposition}{2}
\begin{proposition}
\label{prop:5.3}
If $\emptyset \neq \mathcal F \subset \mathcal A_0 \cup \mathcal A_1$ then
\beq
\label{eq:5.3}
\bigl|\partial^j \mathcal F\bigr| / |\mathcal F| \geq {t + 2\choose j + 1} \bigm/(t + 2) \ \ \text{ for } \ 1 < j < t.
\eeq
If $\emptyset \neq \mathcal F \subset \mathcal A_0$ then
\beq
\label{eq:5.4}
\bigl|\partial^j \mathcal F\bigr| \bigm/ |\mathcal F| \geq \bigl|\partial^j \mathcal A_0\bigr|\bigm/ |\mathcal A_0| > {t\choose j}.
\eeq
\end{proposition}

\begin{proof}
To prove \eqref{eq:5.3} just note that $w_t(\mathcal F) \leq w_t(\mathcal A_0 \cup \mathcal A_1) = 1$.
Now the inequality follows from Theorem \ref{th:2.10}.

To prove \eqref{eq:5.4} we are going to use Proposition \ref{prop:1.1}.

Set $\overline{\mathcal F} = \{F \setminus [t] ; F \in \mathcal F\}$.
Since $\mathcal F \subset \mathcal A_0$, $|\overline{\mathcal F}| = |\mathcal F|$.
For convenience let us introduce the notation $\partial^0\overline{\mathcal F} = \overline{\mathcal F}$, $\partial^1\overline{\mathcal F} = \partial \overline{\mathcal F}$.

\setcounter{claim}{3}
\begin{claim}
\label{cl:5.4}
\beq
\label{eq:5.5}
\bigl|\partial^j \mathcal F\bigr| = \sum_{0 \leq i \leq j} {t\choose j - i} \bigl|\partial^i \overline{\mathcal F}\bigr|.
\eeq
\end{claim}

\begin{proof}
For $0 \leq i \leq j$ define
$$
\mathcal H_i = \bigl\{H \in \partial^j\mathcal F : |H\cap [t]| = i\bigr\}.
$$
That is, $\mathcal H_i$ consists of the $j$'th shadows where we omit $j - i$ elements from $[t]$ and $i$ elements from $F\setminus [t]$.
Then $\bigl|\mathcal H_i\bigr| = {t \choose j - i} \bigl|\partial^i \overline{\mathcal F}\bigr|$.
Since $\partial^j \mathcal F = \mathcal H_0 \sqcup \ldots \sqcup \mathcal H_j$ is a partition, \eqref{eq:5.5} follows.
\end{proof}

Applying \eqref{eq:1.1} to $\overline{\mathcal F}$ and using \eqref{eq:5.5} we infer
\beq
\label{eq:5.6}
\bigl|\partial^j \mathcal F\bigr| \bigm/ |\mathcal F| \geq \sum_{0 \leq i \leq j} {t\choose j - i} {n - t\choose k - t - i} \biggm/ {n - t\choose k - t}.
\eeq

For the family $\mathcal A_0$, $\overline{\mathcal A_0} = {[t + 1, n]\choose k - t}$.
Thus $\bigl|\partial^i \overline{\mathcal A}_0\bigr| = {n - t\choose k - t - i}$.
Consequently, $\bigl|\partial^j \mathcal A_0\bigr| \bigm/ \bigl|\mathcal A_0\bigr| = \sum\limits_{0 \leq i \leq j} {t\choose j - i} {n - t\choose k - t - i} \Bigm/ {n - t\choose k - t}$.
Comparing with \eqref{eq:5.6} the inequality \eqref{eq:5.4} follows.
\end{proof}

The main result of the present section is the following.

\setcounter{theorem}{4}
\begin{theorem}
\label{th:5.5}
Suppose that $\mathcal F \subset {[n]\choose k}$ is a $t$-intersecting $(t + 1)$-semistar.
Then for all $1 < j < t$, \eqref{eq:5.3} holds.
\end{theorem}

Since $\mathcal A_0 \cup \mathcal A_1$ is a semistar with $D = [t + 1]$, Theorem \ref{th:5.5} generalizes Proposition \ref{prop:5.3}.

\begin{proof}
Without loss of generality let $D = [t + 1]$.
That is, $|F \cap [t + 1]| \geq t$ for all $F \in \mathcal F$.
Since shifting maintains this property and does not increase the shadow, we may assume that $\mathcal F$ is shifted.

Set $\mathcal F_0 = \{F \in \mathcal F : [t + 1] \subset \mathcal F\}$ and $\overline{\mathcal F}_0 = \bigl\{F \setminus [t + 1] : F \in \mathcal F_0\bigr\}$.
Define the restricted shadow $\partial_R^j \mathcal F_0$ by
$$
\partial_R^j \mathcal F_0 = \left\{S \cup T : S \in {[t + 1]\choose t + 1 - j}, T \in \overline{\mathcal F}_0\right\}.
$$
Define next $\mathcal T = \left\{ T \in {[t + 2, n]\choose k - t} : \exists G \in {[t + 1]\choose t}, G \cup T \in \mathcal F \right\}$.
For $T \in \mathcal T$ we define
$$
\mathcal G_T = \left\{ G \in {[t + 1]\choose t} : G \cup T \in \mathcal F\right\}
\ \ \text{ and } \ \
\mathcal F_T = \bigl\{G \cup T : G \in \mathcal G_T\bigr\}.
$$
Since $\mathcal G_T \subset {[t + 1]\choose t}$, \eqref{eq:1.1} yields
\beq
\label{eq:5.7}
\bigl|\partial^j \mathcal G_T\bigr| \geq \bigl|\mathcal G_T\bigr| {t + 1\choose t - j} \biggm/ {t + 1\choose t} = \bigl|\mathcal G_T\bigr| {t + 1\choose j + 1} \biggm/ (t + 1).
\eeq
Let us note that for $T \in \mathcal T$ the families $\mathcal F_T$ partition $\mathcal F \setminus \mathcal F_0$.

Let us divide $\mathcal T$ into two parts, $\mathcal T = \mathcal T_1 \cup \mathcal T_2$ where $\mathcal T_1 = \bigl\{T \in \mathcal T : \bigl|\mathcal G_T\bigr| = 1\bigr\}$,
$\mathcal T_2 = \bigl\{T \in \mathcal T : \bigl|\mathcal G_T\bigr| \geq 2\bigr\}$.
For $T \in \mathcal T_1$ one has $\bigl|\partial^j \mathcal G_T\bigr| = {t\choose j}$.
Setting $\mathcal F_i = \bigcup\limits_{T \in \mathcal T_i} \mathcal G_T$, $i = 1,2$, we have
\beq
\label{eq:5.8}
\bigl|\partial_R^j \mathcal F_1\bigr| = \bigl|\mathcal F_1\bigr| {t \choose j},
\eeq
and using \eqref{eq:5.7}
\beq
\label{eq:5.9}
\bigl|\partial_R^j \mathcal F_2\bigr| \geq \bigl|\mathcal F_2\bigr| \frac{{t + 1 \choose j + 1}}{t + 1}.
\eeq
Note that ${t\choose j}$ is larger than the coefficient in \eqref{eq:5.3}.
Indeed,
$$
\frac{{t + 2\choose j + 1}}{t + 2} = \frac{{t + 1\choose j}}{j + 1} = \frac{t + 1}{(j + 1)(t - j + 1)} {t\choose j} < {t \choose j}.
$$
From \eqref{eq:5.8}, \eqref{eq:5.9} and the obvious formula $\bigl|\partial_R^j \mathcal F_0\bigr| = \bigl|\mathcal F_0\bigr| {t + 1\choose j}$ we infer
\beq
\label{eq:5.10}
\bigl|\partial^j \mathcal F\bigr| \geq \sum_{0 \leq i \leq 2} \bigl|\partial_R^j \mathcal F_i\bigr| \geq \bigl|\mathcal F_0\bigr| {t + 1\choose j} + \bigl|\mathcal F_1\bigr| \frac{{t + 2\choose j + 1}}{t + 2} + \bigl|\mathcal F_2\bigr| \frac{{t + 1\choose j + 1}}{t + 1}.
\eeq

To conclude the proof we need a relation between $\mathcal F_0$ and $\mathcal F_2$.

\setcounter{claim}{5}
\begin{claim}
\label{cl:5.6}
$(t + 1) \cdot \bigl|\mathcal F_0\bigr| \geq \bigl|\mathcal F_2\bigr|$.
\end{claim}

\begin{proof}[Proof of the Claim]
First we show that $\mathcal T_2$ is intersecting.
Indeed, if $T \in \mathcal T_2$ then there are at least two choices of $G \in {[t + 1]\choose t}$, $G \in \mathcal G_T$.
Thus for $T, T' \in \mathcal T_2$ we can choose \emph{distinct} $G, G' \in {[t + 1]\choose t}$ so that $G \cup T$, $G' \cup T' \in \mathcal F$.
Now $\bigl|(G \cup T) \cap (G' \cup T')\bigr| = t - 1 + |T \cap T'|$.
Since $\mathcal F$ is $t$-intersecting, $T \cap T' \neq \emptyset$.

Applying Theorem \ref{th:1.3} to $\mathcal T_2$ yields $\bigl|\partial \mathcal T_2\bigr| \geq \bigl|\mathcal T_2\bigr|$.
The inequality $\bigl|\mathcal F_2\bigr| \leq (t + 1) \bigl|\mathcal T_2\bigr|$ should be obvious.
To conclude the proof of the claim let us show
$$
\bigl|\mathcal F_0\bigr| \geq \bigl|\partial \mathcal T_2\bigr|.
$$

More is true. Namely
\beq
\label{eq:5.11}
\overline{\mathcal F}_0 \supset \partial \mathcal T.
\eeq
To prove \eqref{eq:5.11} pick an arbitrary $V \in \partial \mathcal T$.
Then we can choose $G \in {[t + 1]\choose t}$, $T \in \mathcal T$ and $x \in T$ so that $V = T \setminus \{x\}$ and $G \cup T \in \mathcal F$.
Let $y$ be the unique element in $[t + 1] \setminus G$.
Obviously $y < x$.
Thus $[t + 1] \cup V \prec G \cup T$ whence $[t + 1] \cup V \in \mathcal F$.
That is, $V \in \overline{\mathcal F}_0$.
\end{proof}

Now let us rewrite \eqref{eq:5.10}:
$$
\bigl|\partial^j \mathcal F\bigr| \geq |\mathcal F| \frac{{t + 2\choose j + 1}}{t + 2} + \left\{ \bigl|\mathcal F_0\bigr|
 \left({t + 1\choose j} - \frac{{t + 2\choose j + 1}}{t + 2}\right) - \bigl|\mathcal F_2\bigr| \left(\frac{{t + 2\choose j + 1}}{t + 2} - \frac{{t + 1\choose j + 1}}{t + 1}\right)\right\}.
$$
By Claim \ref{cl:5.6} the quantity in $\bigl\{\quad\bigr\}$ is at least
\begin{gather*}
\bigl|\mathcal F_0\bigr|\left({t + 1\choose j} - \frac{{t + 2\choose j + 1}}{t + 1}\right) - (t + 1) \left(\frac{{t + 2\choose j + 1}}{t + 2} - \frac{{t + 1\choose j + 1}}{t + 1}\right)\\
= \bigl|\mathcal F_0\bigr| \left({t + 1\choose j} - {t + 2\choose j + 1} + {t + 1\choose j + 1} \right) = 0,
\end{gather*}
completing the whole proof.
\end{proof}

\section{On the structure and shadow of very large families}
\label{sec:6}

Throughout this section $\mathcal F \subset {[n]\choose k}$ is shifted and $t$-intersecting.
We assume also that $n \geq (k - t + 1)(t + 1)$ which guarantees by Theorem \ref{th:5.1} (Full Erd\H{o}s--Ko--Rado Theorem) that $|\mathcal F| \leq \bigl|\mathcal A_0\bigr|$.

Since $\mathcal A_0$ is a $t$-star, it is natural to investigate the maximum of $\mathcal F$ assuming $\mathcal F \not\subset \mathcal A_0$, i.e., $\mathcal F$ is not a $t$-star.
Of course, $\mathcal A_1$ is a strong candidate, but there is an other one.

\begin{definition}
\label{def:6.1}
Define
$\mathcal H = \mathcal H(n, k, t) = \Bigl\{H \in {[n]\choose k} : [t] \subset H, H \cap [t + 1, k + 1] \neq \emptyset\Bigr\} \cup \Bigl\{[k + 1] \setminus \{x\} : x \in [t]\Bigr\}$.
\end{definition}

\setcounter{theorem}{1}
\begin{theorem}[Hilton--Milner--Frankl Theorem]
\label{th:6.2}
Let $n \geq (k - t + 1)(t + 1)$.
Suppose that $\mathcal F \subset {[n]\choose k}$ is $t$-intersecting but $\mathcal F$ is not a $t$-star.
Then
\beq
\label{eq:6.1}
|\mathcal F| \leq \max\bigl\{\bigl|\mathcal A_1\bigr|, |\mathcal H|\bigr\}.
\eeq
Moreover, except for the case $(n, k, t) = (2k, k, 1)$ equality holds only if $\mathcal F$ is isomorphic to $\mathcal A_1$ or $\mathcal H$.
\end{theorem}

The case $t = 1$ was proved by Hilton and Milner (\cite{HM}).
There have been various shorter proofs given cf.\ \cite{FF2}, \cite{KZ}, \cite{HK} or \cite{F19}.
The case of $t \geq 15$ was proved in \cite{F78}, cf.\ also \cite{F78b}.
Ahlswede and Khachatrian \cite{AK2} gave a different proof valid for the full range.

One should note that for $t + 2 > k - t + 1$, i.e., $k \leq 2t$, $|\mathcal A_1| > |\mathcal H|$.
This implies

\setcounter{corollary}{2}
\begin{corollary}
\label{cor:6.3}
Suppose that $n \geq (k - t + 1)(t + 1)$, $k \leq 2t$, $t > j \geq 1$.
Let $\mathcal F \subset {[n]\choose k}$ be $t$-intersecting and $|\mathcal F| > |\mathcal A_1|$.
Then
\beq
\label{eq:6.2}
\bigl|\partial^j \mathcal F\bigr| \bigm/ |\mathcal F| \geq \bigl|\partial^j \mathcal A_0\bigr| \bigm/ \bigl|\mathcal A_0\bigr| > {t\choose j}.
\eeq
\end{corollary}

Our aim is to prove a similar result for the case $k > 2t$ as well.

We need quite some preparation.
Let us recall a structural result from \cite{F87}.
For a shifted $t$-intersecting family $\mathcal F \subset {[n]\choose k}$ define its base $\mathcal B = \mathcal B(\mathcal F)$ by
$$
\mathcal B = \bigl\{ F \cap [2k - t] : F \in \mathcal F\bigr\}.
$$
Define $\mathcal B^{(\ell)} = \{B \in \mathcal B : |B| = \ell\}$, $b_\ell = \bigl|\mathcal B^{(\ell)}\bigr|$.

\setcounter{proposition}{3}
\begin{proposition}[\cite{F87}]
\label{prop:6.4}
{\rm (i) $\sim$ (iv)} hold.

\hspace*{6.8pt}{\rm (i)} $\mathcal B$ is shifted and $t$-intersecting.

\hspace*{3.4pt}{\rm (ii)} $b_\ell = 0$ for $\ell < t$.

{\rm (iii)} $b_t \leq 1$ with $b_t = 1$ implying that $\mathcal F$ is a $t$-star.

{\rm (iv)} $|\mathcal F| \leq \sum\limits_{t \leq \ell \leq k} b_\ell {n - 2k + t\choose k - \ell}$.
\end{proposition}

Let us mention that using Theorem \ref{th:1.3} (i) implies $\bigl|\partial^t \mathcal B^{(\ell)}\bigr| \geq \bigl|\mathcal B^{(\ell)}\bigr|$.
Since $\partial^t \mathcal B^{(\ell)} \subset {[2k - t]\choose \ell - t}$,
\beq
\label{eq:6.3}
b_\ell \leq {2k - t\choose \ell - t}.
\eeq

For $\ell = t + 1$ one can analyze the possible structure of $\mathcal B^{(\ell)}$.
Note that $[t + 1] \prec [t] \cup \{t + 2\}$ are the two smallest $(t + 1)$-sets in the shifting partial order.
The third ex aequo are $A_3 = [t + 2] \setminus \{t\}$ and $D_3 = [t] \cup \{t + 3\}$.

\setcounter{claim}{4}
\begin{claim}
\label{cl:6.5}
If $A_3 \in \mathcal B^{(t + 1)}$ then $\mathcal F \subset \mathcal A_1$.
\end{claim}

\begin{proof}
We must show $|F \cap [t + 2]| \geq t + 1$.
If this fails then using shiftedness we can find $F$ with $F \cap [t + 2] = [t]$.
This implies $F \cap A_3 = [t - 1]$ contradicting Proposition \ref{prop:6.4}~(i).
\end{proof}

From now on throughout this section we suppose $\mathcal F \not\subset \mathcal A_1$ and thereby $A_3 \notin \mathcal B^{(t + 1)}$.

\begin{claim}
\label{cl:6.6}
If $A_3 \notin \mathcal B^{(t + 1)}$ then $\mathcal B^{(t + 1)} = \left\{[t] \cup \{x\} : t + 1 \leq x \leq t + b_{t + 1}\right\}$.
\end{claim}

\begin{proof}
The statement is trivially true by shiftedness for $b_{t + 1} = 0$, $1$ or $2$.
Suppose $b_{t + 1} \geq 3$.
Then $D_3 \in \mathcal B^{(t + 1)}$.
We claim that $[t] \subset B$ for all $B \in \mathcal B^{(t + 1)}$.

Set $D_i = [t] \cup \{t + i\}$ for $i = 1,2$.
By $D_1 \prec D_2 \prec D_3$, all three are in $\mathcal B^{(t + 1)}$.
In view of Proposition \ref{prop:6.4} (i), $|B \cap D_i| \geq t$, $i = 1,2,3$, implying $[t] \subset B$.
Now Claim \ref{cl:6.6} follows by shiftedness.
\end{proof}

Now we are ready to state and prove the main result of this section.

\setcounter{theorem}{6}
\begin{theorem}
\label{th:6.7}
Suppose that $\mathcal F \subset {[n]\choose k}$ is shifted, $t$-intersecting, $\mathcal F \not\subset \mathcal A_1$ and $b^{(t + 1)} \geq t + 1$.
Then
\beq
\label{eq:6.4}
\bigl|\partial^j \mathcal F\bigr| > {t\choose j} |\mathcal F|.
\eeq
\end{theorem}

\begin{proof}
For simpler notation set $s = b_{t + 1}$.
If $\mathcal F \subset \mathcal A_0$, then \eqref{eq:6.4} is evident.
Suppose that $\mathcal F \not\subset \mathcal A_0$.

\setcounter{claim}{7}
\begin{claim}
\label{cl:6.8}
If $F \in \mathcal F \setminus \mathcal A_0$ then
\beq
\label{eq:6.5}
F \cap [t + s] = [t + s] \setminus \{y\} \ \ \ \text{ for some } \ \ y \in [t].
\eeq
\end{claim}

\begin{proof}
In view of Claim \ref{cl:6.6}, $\mathcal B^{(t + 1)} = \bigl\{[t] \cup \{x\} : t < x \leq t + s\bigr\}$.
By Proposition \ref{prop:6.4} (i) $|F \cap B| \geq t$ for all $B \in \mathcal B^{(t + 1)}$.
Since $[t] \not\subset F$, $x \in F$ for all $t < x \leq t + s$ and $|F \cap [t]| = t - 1$.
\end{proof}

Define $\mathcal F_1 = \bigl\{F \in \mathcal F : |F \cap [t + s]| = t + s - 1\bigr\}$.
In view of Claim \ref{cl:6.8}, $\mathcal F \setminus \mathcal A_0 \subset \mathcal F_1$.
Setting $\mathcal F_0 = \mathcal F \setminus \mathcal F_1$, $\mathcal F_0 \subset \mathcal A_0$ follows.
Defining the restricted shadow with respect to $[t + s]$ as
$$
\partial_R^j \mathcal F = \bigcup_{F \in \mathcal F} \partial_R^j F \ \ \ \text{ where } \ \ \
\partial_R^j F = \left\{S \in {F \choose k - j} : S \setminus [t + s] = F\setminus [t + s]\right\},
$$
it should be clear that $|F \cap [t + s]| \neq |F' \cap [t + s]|$ implies $\partial_R^j F \cap \partial_R^j F' = \emptyset$.
Consequently,
\beq
\label{eq:6.5masodik}
\partial_R^j \mathcal F_0 \cap \partial_R^j \mathcal F_1 = \emptyset.
\eeq
For $\mathcal F_0$, $\mathcal F_0 \subset \mathcal A_0$ implies
\beq
\label{eq:6.6}
\bigl|\partial_R^j \mathcal F_0\bigr| > {t \choose j} \bigl|\mathcal F_0\bigr|.
\eeq

To deal with $\mathcal F_1$ define $\mathcal T \subset {[t + s + 1, n]\choose k - t - s + 1}$ by
$$
\mathcal T = \bigl\{F \setminus [t + s] : F \in \mathcal F_1\bigr\}.
$$

For $T \in \mathcal T$ define $\mathcal G_T = \left\{G \in {[t + s]\choose t + s - 1}: G \cup T \in \mathcal F_1 \right\}$.

Now \eqref{eq:1.1} implies
$$
\bigl|\partial^j \mathcal G_T\bigr| \geq \bigl|\mathcal G_T\bigr| \frac{{t + s\choose t + s - 1 - j}}{{t + s\choose 1}} = \bigl|\mathcal G_T\bigr| \frac{{t + s\choose j + 1}}{t + s}.
$$
By definition
$$
\bigl|\mathcal F_1\bigr| = \sum_{t \in \mathcal T} \bigl|\mathcal G_T\bigr| \ \ \ \text{ and } \ \ \
\bigl|\partial_R^j\mathcal F_1\bigr| = \sum_{t \in \mathcal T} \bigl|\partial^j\mathcal G_T\bigr| .
$$
Consequently,
\beq
\label{eq:6.7}
\bigl|\partial_R^j \mathcal F_1 \bigr| \geq \bigl|\mathcal F_1\bigr| \frac{{t + s\choose j + 1}}{t + s}.
\eeq
Let us show that $s \geq t + 1$ implies
$$
\frac{{t + s\choose j + 1}}{t + s} = \frac{{t + s - 1\choose j}}{j + 1} \geq \frac{{2t\choose j}}{j + 1} \geq {t \choose j}.
$$
Indeed,
$$
\frac{{2t\choose j}}{{t \choose j}} = \prod_{0 \leq i < j} \frac{2t - i}{t - i} \geq 2^j \geq j + 1.
$$
Thus adding \eqref{eq:6.6} and \eqref{eq:6.7}, and using \eqref{eq:6.5masodik} imply \eqref{eq:6.4}.
\end{proof}

\begin{rema}
For $j = 1$, $2^1 = 1 + 1$.
However for larger values of $j$ one can considerably relax the condition $b_{t + 1} \geq t + 1$.
\end{rema}

\setcounter{corollary}{7}
\begin{corollary}
\label{cor:6.8}
Suppose that $\mathcal F \subset {[n]\choose k}$ is shifted, $t$-intersecting,
$\mathcal F \not\subset \mathcal A_1$, $t + 2 \leq k - t + 1$.
If
$$
|\mathcal F| > t {n - 2k + t\choose k - t - 1} + \sum_{t + 2 \leq \ell \leq k} {2k - t\choose \ell - t} {n \choose k - \ell}
$$
then
\beq
\label{eq:6.8}
\bigl|\partial^j \mathcal F\bigr| > {t\choose j} |\mathcal F|.
\eeq
\end{corollary}

\begin{proof}
If $\mathcal F \subset \mathcal A_0$ then \eqref{eq:6.8} is evident.
Otherwise $b_t = 0$ and thereby $b_{t + 1} \geq t + 1$ follow from Proposition \ref{prop:6.4}.
Now \eqref{eq:6.8} is a consequence of Theorem \ref{th:6.7}.
\end{proof}

\section{A general bound}
\label{sec:7}

To make notation simpler let us define $\gamma(\ell, t, j) = {t + 2(\ell - t)\choose t + \ell - j} \Bigm/ {t + 2(\ell - t)\choose t + \ell}$.
Consider a shifted $t$-intersecting family $\mathcal F \subset {[n]\choose k}$.
Recall the definition of $w_t(\mathcal F)$ as the minimal integer $w$, $0 \leq w \leq k - t$ such that for every $F \in \mathcal F$ there exists $\ell = \ell(F)$, $0 \leq \ell \leq w$, with
\beq
\label{eq:7.1}
\bigl|F \cap [t + 2\ell]\bigr| \geq t + \ell.
\eeq
Since \eqref{eq:7.1} holds with $i$ for $F \in \mathcal A_i$, $w_t(\mathcal F) > i$ implies $\mathcal F \not\subset \mathcal A_0 \cup \ldots \cup \mathcal A_i$.

Suppose that
\beq
\label{eq:7.2}
\bigl|\partial^j \mathcal F\bigr| \bigm/ |\mathcal F| < \gamma(w, t, j).
\eeq
By Theorem \ref{th:2.10}, $\mathcal F \not\subset \mathcal A_0 \cup \ldots \cup \mathcal A_w$.
That is, we can find some $F \in \mathcal F$ failing \eqref{eq:7.1} for all $0 \leq \ell \leq w$.

Define $E = [t - 1] \cup (t + 1, t + 3, \ldots, t + 2w + 1) \cup [t + 2w + 2, k + w + 1]$.
Then $E \prec F$ and by shiftedness $E \in \mathcal F$.

Define $D = [t] \cup (t + 2, \dots, t + 2w) \in {[2k - t]\choose t + w}$.
Note that $|E \cap D| = t - 1$.
This permits to prove

\begin{proposition}
\label{prop:7.1}
If $G \in \mathcal F$ then either {\rm (i)} or {\rm (ii)} hold.

\hspace*{3.4pt}{\rm (i)} $\bigl|G \cap [t + 1 + 2h]\bigr| \geq t + 1 + h$ for some $0 \leq h < w$.

\smallskip
{\rm (ii)} $\bigl|G \cap [2k - t]\bigr| > w + t$.
\end{proposition}

\begin{proof}
Suppose that (ii) does not hold.
Let $|G \cap [2k - t]| = t + h$ for some $0 \leq h \leq w$.
Set $D_h = D \cap [t + 2h]$.
Since $\bigl|E \cap D_h\bigr| = t - 1$, we infer $D_h \not\prec G \cap [2k - t]$ by shiftedness and Proposition \ref{prop:6.4}.
Thus (i) follows.
\end{proof}

Define the partition $\mathcal F = \mathcal F_{\text{in}} \cup \mathcal F_{\text{out}}$ by
\begin{align*}
\mathcal F_{\text{in}} &= \bigl\{F \in \mathcal F : |F \cap [2k - t] | > w + t\bigr\},\\
\mathcal F_{\text{out}} &= \bigl\{F \in \mathcal F : |F \cap [2k - t] | \leq w + t\bigr\}.
\end{align*}

With the definition of restricted $j$-shadows as in Definition~\ref{def:2.6} we have
\beq
\label{eq:7.3}
\bigl|\partial^j \mathcal F\bigr| \geq \bigl|\partial_R^j \mathcal F_{\text{in}}\bigr| + \bigl|\partial_R^j \mathcal F_{\text{out}}\bigr|.
\eeq

In view of Proposition \ref{prop:6.4}, the family $\bigl\{F \cap [2k - t] : F \in \mathcal F\bigr\}$ is $t$-intersecting.
Thus by Theorem \ref{th:1.3} we have
\beq
\label{eq:7.4}
\bigl|\partial_R^j \mathcal F_{\text{in}}\bigr| \geq \gamma(k - t, t, j)\bigl|\mathcal F_{\text{in}}\bigr|.
\eeq
As to $\mathcal F_{\text{out}}$, Proposition \ref{prop:7.1} (i) implies that it is pseudo $t + 1$-intersecting with $w_{t + 1}\bigl(\mathcal F_{\text{out}}\bigr) \leq w - 1$.
By Theorem \ref{th:2.10} we have
\beq
\label{eq:7.5}
\bigl|\partial_R^j \mathcal F_{\text{out}}\bigr| \geq \gamma(w - 1, t + 1, j) \bigl|\mathcal F_{\text{out}}\bigr|.
\eeq
Defining $\alpha$, $\beta$ by
$$
\alpha = \gamma(w, t, j) - \gamma(k - t, t, j) \ \ \ \text{ and } \ \ \ \beta = \gamma(w - 1, t + 1, j) - \gamma(w, t, j)
$$
we infer from \eqref{eq:7.3}, \eqref{eq:7.4} and \eqref{eq:7.5}
$$
\bigl|\partial^j \mathcal F\bigr| \geq \gamma(w, t, j) |\mathcal F| + \beta \bigl|\mathcal F_{\text{out}}\bigr| - \alpha \bigl|\mathcal F_{\text{in}}\bigr|.
$$
Thus we proved

\begin{proposition}
\label{prop:7.2}
If $\bigl|\mathcal F_{\text{out}}\bigr| \geq \frac{\alpha}{\beta} \bigl|\mathcal F_{\text{in}}\bigr|$ then
\beq
\label{eq:7.6}
\bigl|\partial^j \mathcal F\bigr| \geq \gamma(w, t, j)|\mathcal F|.
\eeq
\end{proposition}

Note that $\alpha$ and $\beta$ are independent of $n$, that is, $\frac{\alpha}{\beta}$ is a constant.
Also, to bound $\bigl|\mathcal F_{\text{in}}\bigr|$ we may use \eqref{eq:6.3} and Proposition \ref{prop:6.4}:
$$
\bigl|\mathcal F_{\text{in}}\bigr| \leq \sum_{w < \ell \leq k - t} {2k - t\choose \ell} {n - 2k + t\choose k - \ell - t} = (1 + o(1)) {2k - t\choose w + 1}{n - 2k + t\choose k - w - t - 1}.
$$
If \eqref{eq:7.6} fails then
\begin{align*}
\bigl|\mathcal F_{\text{out}}\bigr| &< \left(\frac{\alpha}{\beta} + o(1)\right) {2k - t\choose w + 1} {n - 2k + t\choose k - w - t - 1}, \ \text{ i.e.,}\\
|\mathcal F| &< \frac{\alpha + \beta + o(1)}{\alpha} {2k - t\choose w + 1} {n - 2k + t\choose k - w - t - 1}.
\end{align*}
That is, we proved the following

\setcounter{theorem}{2}
\begin{theorem}
\label{th:7.3}
Suppose that $\mathcal F \subset {[n]\choose k}$ is $t$-intersecting,
\beq
\label{eq:7.7}
|\mathcal F| > \frac{\alpha + \beta + o(1)}{\alpha} {2k - t\choose w + 1} {n - 2k + t\choose k - w - t - 1}.
\eeq
Then
\beq
\label{eq:7.8}
\bigl|\partial^j \mathcal F\bigr| \geq \gamma(w, t, j)|\mathcal F|.
\eeq
\end{theorem}

There are many ways that Theorem \ref{th:7.3} can be improved.
The simplest is to replace ${2k - t\choose w + 1}$ by ${2k - t - 1\choose w + 1}$ unless $w = k - t$ (the case that we treated in Theorem \ref{th:1.4}).
More substantial is the improvement that except for the part of $\mathcal F_{\text{in}}$ contained in $\mathcal A_{\ell + 1} \cup \mathcal A_{\ell + 2} \cup \ldots \cup \mathcal A_{k - t}$ one can replace the factor $\gamma(k - t, t, j)$ in \eqref{eq:7.4} by the larger $\gamma(\ell, t, j)$ leading to a considerably smaller value of~$\alpha$.

For $n \to \infty$, $\bigl|\mathcal A_{i + 1}\bigr| = O\bigl(\bigl|\mathcal A_i \bigr| / n\bigr)$ showing that asymptotically only \hbox{$\gamma(w + 1, t, j)$} matters.
That is, Theorem \ref{th:7.3} holds with $\alpha = \gamma(w + 1, t, j) + \varepsilon$ for any $\varepsilon > 0$ and $n > n_0(\varepsilon)$.

Let us close the paper by an open problem.

\setcounter{problem}{3}
\begin{problem}
\label{pr:7.4}
Determine or estimate the smallest value of $c = c(k, t, j)$ such that \eqref{eq:7.8} holds whenever $n > n_0(k, t, j)$ and $|\mathcal F| > c {n\choose k - w - t - 1}$.
\end{problem}

%%\smallskip
%%\noindent{\bf Akcnowledgement.}

\end{document}